%% file: On_the_stability_of_orthogonally_Jensen_quadratic_functional_equation.tex
\documentclass[10pt,reqno,final]{article}

\usepackage{amsmath,amsfonts,amssymb,amsthm,version}
\usepackage{mathrsfs,fancybox,pifont}
\usepackage{graphicx}
\usepackage{url,hyperref}
\usepackage[notcite,notref]{showkeys}
\usepackage{color}
\usepackage{subfigure,multirow}
\usepackage{epstopdf}
\usepackage{cases}
\usepackage{mathtools}
\usepackage{algorithm,algorithmic}
\usepackage{authblk}
\usepackage{lipsum}
\usepackage{geometry}
\usepackage{enumerate}

\allowdisplaybreaks

\numberwithin{equation}{section}
\numberwithin{figure}{section}
\numberwithin{table}{section}

\theoremstyle{plain}
\newtheorem{theorem}{Theorem}[section]
\newtheorem{lemma}{Lemma}[section]
\newtheorem{corollary}{Corollary}[section]

\theoremstyle{definition}
\newtheorem{definition}{Definition}[section]

\theoremstyle{remark}

\title{On the stability of orthogonally Jensen additive and quadratic functional equation\thanks{The work was supported by the National Natural Science Foundation of P. R. China (Nos. 11971493 and 12071491).} }
\author[1]{Linlin Fu\thanks{fullin3@mail.sysu.edu.cn}}
\author[1]{Qi Liu\thanks{Corresponding author. {\it liuq325@mail2.sysu.edu.cn}}}
\author[1]{Yongjin Li\thanks{stslyj@mail.sysu.edu.cn}}
\affil[1]{Department of Mathematics, Sun Yat-sen University\\ Guangzhou 510275, P.R. China}

\date{}

\begin{document}
	\maketitle

	\begin{abstract}
		We consider the stability of the orthogonal Jensen additive and quadratic equations in $F$-spaces, through applying and extending the approach to the proof of a 2010 result of  W.Frchner and J.Sikorska, we presenting a new method to get the stability. Moreover, we work in a more general and natural condition than considered before by other antuors.


	\end{abstract}

	\input{Introduction}

    \input{JensenAdditive}
    \input{JensenQuadractic}

	\section*{Acknowledgments}
	We would like to acknowledge the assistance of volunteers in putting
	together this example manuscript and supplement.

    \bibliographystyle{listbib}
    \bibliography{reference}

\end{document}

%% file: Introduction.tex
\section{Introduction and preliminaries}

\quad Assume that $X$ is a real inner product space and $f: X \rightarrow \mathbb{R}$ is a solution of the orthogonal Cauchy functional equation $f(x+y)=f(x)+f(y),\, \langle x, y\rangle=0$. By the Pythagorean theorem $f(x)=\|x\|^{2}$ is a solution of the conditional equation. Of course, this function does not satisfy the additivity equation everywhere. Thus orthogonal Cauchy equation is not equivalent to the classic Cauchy equation on the whole inner product space. This phenomenon may show the significance of study of orthogonal Cauchy equation.

In the recent decades, stability of functional equations have been investigated by many mathematicians (see \cite{1995cfea}). The first author treating the stability of the Cauchy equation was D.H. Hyers \cite{1941otso} by proving that if $f$ is a mapping from a normed space $X$ into a Banach space satisfying $\| f(x+y)-$ $f(x)-f(y) \| \leq \epsilon$ for some $\epsilon>0$, then there is a unique additive mapping $g: X \rightarrow Y$ such that $\|f(x)-g(x)\| \leq \epsilon$.

R. Ger and J. Sikorska \cite{1995soto} investigated the orthogonal stability of the Cauchy functional equation $f(x+y)=f(x)+f(y)$, namely, they showed that if $f$ is a function from an orthogonality space $X$ into a real Banach space $Y$ and $\|f(x+y)-f(x)-f(y)\| \leq \epsilon$ for all $x, y \in X$ with $x \perp y$ and some $\epsilon>0$, then there exists exactly one orthogonally additive mapping $g: X \rightarrow Y$ such that $\|f(x)-g(x)\| \leq \frac{16}{3} \epsilon$ for all $x \in X$.

The first author treating the stability of the quadratic equation was $\mathrm{F}$. Skof \cite{1983plea} by proving that if $f$ is a mapping from a normed space $X$ into a Banach space $Y$ satisfying $\|f(x+y)+f(x-y)-2 f(x)-2 f(y)\| \leq \epsilon$ for some $\epsilon>0$, then there is a unique quadratic function $g: X \rightarrow Y$ such that $\|f(x)-g(x)\| \leq \frac{\epsilon}{2}$. P. W. Cholewa \cite{1984rots} extended Skof's theorem by replacing $X$ by an abelian group G. Skof's result was later generalized by S. Czerwik \cite{1992otso} in the spirit of Hyers-Ulam. The stability problem of functional equations has been extensively investigated by some mathematicians (see \cite{2002feai,2003sofe,1998sofe,1978otso,2000tpos,2000otso,2003feia,1960pimm}).

The orthogonally quadratic equation
\begin{equation*}
    f(x+y)+f(x-y)=2 f(x)+2 f(y), x \perp y
\end{equation*}
was first investigated by $\mathrm{F}$. Vajzović \cite{1967udfh} when $X$ is a Hilbert space, $Y$ is the scalar field, $f$ is continuous and $\perp$ means the Hilbert space orthogonality. Later, H. Drljević \cite{1986oafw}, M. Fochi \cite{1989feoa}, M. Moslehian \cite{2006otso,2005otos} and G. Szabó \cite{1990soqm} generalized this result.

The notion of orthogonality goes a long way back in time and various extensions have been introduced over the last decades. In particular, proposing the notion of orthogonality in normed linear spaces has been the object of extensive efforts of many mathematicians.

Let us recall the standard definition of orthogonality by R\"{a}tz \cite{1985ooam}:
\begin{definition}
    Let $X$ be a real linear space with $\operatorname{dim}X \geq  2$ and let $\perp$ be a binary relation
    on $X$ such that

    {\rm (1)} $x\perp 0$ and $0\perp x$ for all $x\in X$;

    {\rm (2)}  if $x,y\in X \backslash \{0\}$  and $x\perp y$, then $x$ and $y$ are linearly independent;

    {\rm (3)}	 if $x,y\in X$ and $x\perp y$, then for all $\alpha,\beta \in \mathbb{R}$ we have $\alpha x\perp \beta y$;

    {\rm (4)}	  for any two-dimensional subspace $P$ of $X$ and for every $x \in  P$ , $\lambda \in [0, \infty)$,
    there exists $y\in P$ such that $x\perp y$ and $x+y\perp \lambda x-y$.

    An ordered pair $(X,\perp)$ is called an orthogonality space.

\end{definition}

In 2010, Fechner and Sikorska \cite{2010otso} studied the stability of orthogonality and proposed the definition of orthogonality as follows.

\begin{definition}\label{dt2}
    Let $X$ be an Abelian group and let $\perp$ be a binary relation defined on $X$ with the properties:

    {\rm (1)} if $x,y\in X$ and $x\perp y$, then $x\perp -y,-x\perp y$ and $2x\perp 2y$;

    {\rm (2)} for every $x\in X$, there  exists a $y \in X$ such that $x\perp y$ and $x+y\perp x-y$.

\end{definition}

It's worth noting that every orthogonal space satisfies these conditions   as well as any normed linear space with the isosceles orthogonality, but Pythagorean orthogonality no longer satisfies these conditions.

Although various studies on stability have been successfully conducted, there are not many corresponding stability results due to the non-linear structure of the infinite-dimensional $F$-space.
The nonlinear structure of $F$-space plays an important role in functional analysis and other mathematical fields.
The $L^p([0,1])$ for $0<p<1$ equipped with the metric $d(f,g)=\int \vert f(x)-g(x) \vert^pdx$
is an example of an  $F$-space  but not a Banach space.
Besides these, for $F$-spaces, and we recommend readers to read the literature \cite{2008nsos,1981cwzd}.

\begin{definition}\label{dt1}
    Consider $X$ be a linear space. A non-negative valued function $\|\cdot\|$ achieves
    an $F$-norm if satisfies the following conditions:

    {\rm (1)} $\|x\|= 0$ if and only if $x=0$;

    {\rm (2)} $\|\lambda x\|=\|x\|$ for all  $\lambda$,  $|\lambda|=1$;

    {\rm (3)} $\|x+y\|\leq\|x\|+\|y\|$ for all $x,y\in X$;

    {\rm (4)} $\|\lambda_{n}x\|\rightarrow 0$ provided  $\lambda_{n}\rightarrow 0$;

    {\rm (5)} $\|\lambda x_{n}\|\rightarrow 0$ provided  $x_{n}\rightarrow 0$;

    {\rm (6)} $\|\lambda_{n} x_{n}\|\rightarrow 0$ provided  $\lambda_{n}\rightarrow 0, x_{n}\rightarrow 0$.

\end{definition}
{ Then $(X,\|\cdot\|)$ is called an $F^\ast$-space. An F-space is a complete $F^\ast$-space}.

An $F$-norm  is called $\beta$-homogeneous $(\beta>0)$ if $\|tx\|=|t|^{\beta}\|x\|$ for all $x\in X$ and all $t\in \mathcal{C}$ (see \cite{1985mls,1979mmit}).

If a quasi-norm  is $p$-subadditive, then it is called $p$-norm $(0 < p < 1)$.
In other words, if it satisfies
\begin{equation*}
    \Vert x+y\Vert^p\leq \Vert x\Vert^p+\Vert y\Vert^p,~~x,y\in X.
\end{equation*}

We note that the $p$-subadditive quasi-norm  $\Vert \cdot\Vert $ induces an $F$ norm.
We refer the reader to \cite{2008nsos} and  \cite{1984afs} for background on it.

\begin{definition}\rm{\cite{2003foca}}
    A quasi–norm on $\Vert \cdot\Vert$ on vector space $X$ over a field $K(
    \mathbb{R})$ is a map $X\longrightarrow[0, \infty)$ with the following properties:

    ${\text { (1) }\|x\|=0 \text { if and only if } x=0};$

    ${\text { (2) }\|a x\|=|a|\|x\|,~~ a \in \mathbb{R}, x \in X};$

    ${\text { (3) }\|x+y\| \leq C(\|x\|+\|y\|),}~~ {x, y \in X}.$

    where $C \geq 1$ is a constant independent of $x, y \in X$. The smallest $C$ for which
    $(3)$ holds is called the quasi-norm constant of $(X,\Vert \cdot\Vert)$.
\end{definition}

It is worth to note that the well-known Aoki–Rolewicz theorem \cite{1985mls} in nonlocally convex theory, that is,
 which asserts that for some $0 < p\leq 1$, every quasi-norm
admits  an equivalent $p$-norm.

Various more results for the stability of  functional equations in quasi-Banach spaces
can be seen in  \cite{2018tgho,2008iami}. However, the results are more interesting and meaningful
when orthogonality is taken into account.

Let $X$ be an orthogonality space and $Y$ a real Banach space. A mapping $f: X \rightarrow Y$ is called orthogonally Jensen additive if it satisfies the so-called orthogonally Jensen additive functional equation

\begin{equation}\label{Jensen addivtive}
    2 f\left(\frac{x+y}{2}\right)=f(x)+f(y)
\end{equation}
for all $x, y \in X$ with $x \perp y$. A mapping $f: X \rightarrow Y$ is called orthogonally Jensen quadratic if it satisfies the so-called orthogonally Jensen quadratic functional equation

\begin{equation}\label{Jensen quadratic}
    2 f\left(\frac{x+y}{2}\right)+2 f\left(\frac{x-y}{2}\right)=f(x)+f(y)
\end{equation}
for all $x, y \in X$ with $x \perp y$.

In this paper, we apply some ideas and extend the results from \cite{2010otso} to prove the stability of the orthogonally Jensen additive functional equation (\ref{Jensen addivtive}) and of the orthogonally Jensen quadratic functional equation (\ref{Jensen quadratic}) in $F$-spaces and quasi-Banach spaces.

%% file: JensenAdditive.tex
\section{Stability of the Jensen additive functional equations}
\quad Applying some ideas from \cite{2010otso}, we deal with the conditional stability problem for (\ref{Jensen addivtive}).

First, we give the following lemma which is important for our main results in this paper.

\begin{lemma}\label{the key lemma}
    Let $X$ be an Abelian group, and $Y$ be a $\beta$-homogeneous $F$-space.
    For $\varepsilon\geq 0$, assume $f: X\rightarrow Y$ be a mapping such that for all
    $x, y \in X$ and a constant $C>0$ one has
    \begin{equation}\label{183inequation}
        \|f(2x)-\frac{3}{8}f(4x)+\frac{1}{8}f(-4x)\|\leq C.
    \end{equation}
Let $$h(x,n)=\left\|f(2 x)-\frac{2^{n}+1}{2 \cdot 4^{n}} f\left(2^{n+1} x\right)+\frac{2^{n}-1}{2 \cdot 4^{n}} f\left(-2^{n+1} x\right)\right\|$$ and $$g_{n}(x)=\frac{2^{n}+1}{2 \cdot 4^{n}} f\left(2^{n} x\right)-\frac{2^{n}-1}{2 \cdot 4^{n}} f\left(-2^{n} x\right), \quad n \in \mathbb{N}.$$
Then we have that
\begin{enumerate}[(1)]
    \item $|h(x,n+1)-h(x,n)|\leq C\left[ \left( \frac{2^n+1}{2\cdot 4^n} \right) ^{\beta}+\left( \frac{2^n-1}{2\cdot 4^n} \right) ^{\beta} \right]$ and moreover, we have
    \begin{equation}\label{estimation of h(x,n)}
        h(x,n)\leq C
        \left( \sum_{n=1}^{\infty}{\left[ \left( \frac{2^n+1}{2\cdot 4^n} \right) ^{\beta}+\left( \frac{2^n-1}{2\cdot 4^n} \right) ^{\beta} \right]}+1 \right)
    \end{equation}
for all $n\in \mathbb{N}$.
    \item $(g_n(x))_{n\in \mathbb{N}}$ is a Cauchy sequence for every $x\in X$. Hence, the mapping $g : X \rightarrow Y$ can be defined as:
    \begin{equation*}
        g(x):=\lim _{n \rightarrow \infty} g_{n}(x)
    \end{equation*}
and then we have
    \begin{equation*}
        \|f(2x)-g(2x)\|\leq C \left( \sum_{n=1}^{\infty}{\left[ \left( \frac{2^n+1}{2\cdot 4^n} \right) ^{\beta}+\left( \frac{2^n-1}{2\cdot 4^n} \right) ^{\beta} \right]}+1 \right)
    \end{equation*}
for all $x\in X.$
\end{enumerate}
\end{lemma}
\begin{proof}
        First, with the help of (\ref{183inequation}), through a simple estimate we obtain

    $$
    \begin{aligned}&\left\|f(2 x)-\frac{2^{n+1}+1}{2 \cdot 4^{n+1}} f\left(2^{n+2} x\right)+\frac{2^{n+1}-1}{2 \cdot 4^{n+1}} f\left(-2^{n+2} x\right)\right\| \\
        & \leq\left\|f(2 x)-\frac{2^{n}+1}{2 \cdot 4^{n}} f\left(2^{n+1} x\right)+\frac{2^{n}-1}{2 \cdot 4^{n}} f\left(-2^{n+1} x\right)\right\| \\
        & +\bigg(\frac{2^{n}+1}{2 \cdot 4^{n}}\bigg)^\beta\left\|f\left(2^{n+1} x\right)-\frac{3}{8} f\left(2^{n+2} x\right)+\frac{1}{8} f\left(-2^{n+2} x\right)\right\| \\
        & +\bigg(\frac{2^{n}-1}{2 \cdot 4^{n}}\bigg)^\beta\left\|f\left(-2^{n+1} x\right)-\frac{3}{8} f\left(-2^{n+2} x\right)+\frac{1}{8} f\left(2^{n+2} x\right)\right\| \\
        & \leq\left\|f(2 x)-\frac{2^{n}+1}{2 \cdot 4^{n}} f\left(2^{n+1} x\right)+\frac{2^{n}-1}{2 \cdot 4^{n}} f\left(-2^{n+1} x\right)\right\|\\
        &+C\left[ \left( \frac{2^n+1}{2\cdot 4^n} \right) ^{\beta}+\left( \frac{2^n-1}{2\cdot 4^n} \right) ^{\beta} \right]
    \end{aligned}
    $$
    which implies
    $$
    \begin{aligned}&\left\|f(2 x)-\frac{2^{n+1}+1}{2 \cdot 4^{n+1}} f\left(2^{n+2} x\right)+\frac{2^{n+1}-1}{2 \cdot 4^{n+1}} f\left(-2^{n+2} x\right)\right\| \\& -\left\|f(2 x)-\frac{2^{n}+1}{2 \cdot 4^{n}} f\left(2^{n+1} x\right)+\frac{2^{n}-1}{2 \cdot 4^{n}} f\left(-2^{n+1} x\right)\right\|\\
        &\leq C\left[ \left( \frac{2^n+1}{2\cdot 4^n} \right) ^{\beta}+\left( \frac{2^n-1}{2\cdot 4^n} \right) ^{\beta} \right]
    \end{aligned}
    $$
    Now we let
    $$
    h(x,n)=\left\|f(2 x)-\frac{2^{n}+1}{2 \cdot 4^{n}} f\left(2^{n+1} x\right)+\frac{2^{n}-1}{2 \cdot 4^{n}} f\left(-2^{n+1} x\right)\right\|,
    $$
    so we have that
    $$
    |h(x,n+1)-h(x,n)|\leq C\left[ \left( \frac{2^n+1}{2\cdot 4^n} \right) ^{\beta}+\left( \frac{2^n-1}{2\cdot 4^n} \right) ^{\beta} \right],
    $$
    and then
    \begin{align*}
        h(x,n)&=\sum_{i=2}^{n}(h(x,i)-h(x,i-1))+h(x,1)\\
        &\leq C\left( \sum_{n=1}^{\infty}{\left[ \left( \frac{2^n+1}{2\cdot 4^n} \right) ^{\beta}+\left( \frac{2^n-1}{2\cdot 4^n} \right) ^{\beta} \right] +1} \right).
    \end{align*}
    This means that
    \begin{equation*}
        \begin{aligned}
            &\left\|f(2 x)-\frac{2^{n}+1}{2 \cdot 4^{n}} f\left(2^{n+1} x\right)+\frac{2^{n}-1}{2 \cdot 4^{n}} f\left(-2^{n+1} x\right)\right\| \\
            &\leq C\left( \sum_{n=1}^{\infty}{\left[ \left( \frac{2^n+1}{2\cdot 4^n} \right) ^{\beta}+\left( \frac{2^n-1}{2\cdot 4^n} \right) ^{\beta} \right] +1} \right), \quad x\in X.
        \end{aligned}
    \end{equation*}
    The next step is to prove that for each $x \in X$ the sequence
    $$
    g_{n}(x):=\frac{2^{n}+1}{2 \cdot 4^{n}} f\left(2^{n} x\right)-\frac{2^{n}-1}{2 \cdot 4^{n}} f\left(-2^{n} x\right), \quad n \in \mathbb{N}
    $$
    is convergent in $Y$. Since $Y$ is complete, it suffices to show that $(g_n(x))_{n\in \mathbb{N}}$ is
    a Cauchy sequence for every $x \in X$. Applying estimate (\ref{183inequation} ) twice then we have
    \begin{equation*}
        \begin{aligned}\left\|g_{n}(x)-g_{n+1}(x)\right\|=& \bigg\| \frac{2^{n}+1}{2 \cdot 4^{n}}\left(f\left(2^{n} x\right)-\frac{3}{8} f\left(2^{n+1} x\right)+\frac{1}{8} f\left(-2^{n+1} x\right)\right) \\ &-\frac{2^{n}-1}{2 \cdot 4^{n}}\left(f\left(-2^{n} x\right)-\frac{3}{8} f\left(-2^{n+1} x\right)+\frac{1}{8} f\left(2^{n+1} x\right)\right) \bigg\| \\
            \leq &C\left[ \left( \frac{2^n+1}{2\cdot 4^n} \right) ^{\beta}+\left( \frac{2^n-1}{2\cdot 4^n} \right) ^{\beta} \right]
        \end{aligned}
    \end{equation*}
    for each $n \in \mathbb{N}$, which gives us  that $(g_n(x))_{n\in N}$ is a Cauchy sequence.

    Hence, the mapping $g : X \rightarrow Y$ can be defined as:
    \begin{equation*}
        g(x):=\lim _{n \rightarrow \infty} g_{n}(x)
    \end{equation*}
    for all $x \in X$.
    Combining with (\ref{estimation of h(x,n)}) we have
    \begin{equation*}
            \Vert f(2x) - g(2x)\Vert \leq  C\left( \sum_{n=1}^{\infty}{\left[ \left( \frac{2^n+1}{2\cdot 4^n} \right) ^{\beta}+\left( \frac{2^n-1}{2\cdot 4^n} \right) ^{\beta} \right] +1} \right)
    \end{equation*}
    with $x \in X.$
\end{proof}

Next, we would like to show the main result of this section: In this section, let $X$ be an Abelian group and let $\perp$ be a binary relation defined on $X$ with the properties:

($a$) for all $x\in X, 0\perp x$ or $x\perp 0$;

($b$) if $x, y \in X$ and $x\perp y$, then $-x\perp -y$ and $2x\perp 2y$.

\begin{theorem}\label{main theorem for additive}
    Let $X$ be an Abelian group, and $Y$ be a $\beta$-homogeneous $F$-space.
    For $\varepsilon\geq 0$, assume $f: X\rightarrow Y$ be a mapping such that for all
    $x, y \in X$ one has
    \begin{gather}
        x \perp y \quad \text { implies } \quad\|2f\left(\frac{x+y}{2}\right)-f(x)-f(y)\| \leq \varepsilon\label{epsilon Jensen additive}\\
        \text{and }\quad \|f(x)+f(-x)\|\leq \varepsilon. \label{odd inequation condition}
    \end{gather}

    Then there exists a mapping $g: X\rightarrow Y$ such that
    \begin{equation}\label{orth implies Jensen additive}
        x \perp y \quad \text { implies } \quad 2g\left(\frac{x+y}{2}\right)=g(x)+g(y)
    \end{equation}
    and
    \begin{equation}\label{dist(f,g) for Jensen additive}
        \Vert f(x)-g(x)\Vert  \leq \left( \sum_{n=1}^{\infty}{\left[ \left( \frac{2^n+1}{2\cdot 4^n} \right) ^{\beta}+\left( \frac{2^n-1}{2\cdot 4^n} \right) ^{\beta} \right] +1} \right) \cdot \frac{1+2^{\beta}+4^{\beta}+8^{\beta}+16^{\beta}}{8^{\beta}}\cdot \varepsilon
    \end{equation}
    for all $x \in 2X = \{2x : x \in X\}$. Moreover, the mapping $g$ is unique on the
    set $2X$.
\end{theorem}

\begin{proof} It is not hard to get $\|f(0)\|\leq \frac{1}{2^\beta}\varepsilon$ with the help of (\ref{odd inequation condition}).  Then for all $x\in X$, by ($a$), we have $0\perp x$ or $x \perp 0$, so combined with (\ref{epsilon Jensen additive}) we can write
    $$
    \|2f\left(\frac{x}{2}\right)-f(x)\|\leq (\frac{1}{2^\beta}+1)\varepsilon
    $$
    and then we have
    $$
    \|2f\left(x\right)-f(2x)\|\leq (\frac{1}{2^\beta}+1)\varepsilon.
    $$

    Then it is easy to get the following inequation:
    $$
    \begin{aligned}&\| 3 f(4 x)-8 f(2 x)-f(-4 x)) \|\\
        &=\| 4\left[ f\left( 4x \right) -2f\left( x \right) \right] +16\left[ f\left( 2x \right) -2f\left( x \right) \right] +\left[ f\left( -4x \right) +f\left( 4x \right) \right] \| \\
        &\leq 4^\beta\cdot (\frac{1}{2^\beta}+1)\varepsilon+16^\beta\cdot(\frac{1}{2^\beta}+1)\varepsilon+\varepsilon \\
        &=(1+2^\beta+4^\beta+8^\beta+16^\beta)\varepsilon.
    \end{aligned}
    $$
    This proves that
    \begin{equation}\label{183inequ for additive}
            \left\|f(2 x)-\frac{3}{8} f(4 x)+\frac{1}{8} f(-4 x)\right\|
            \leq \frac{(1+2^\beta+4^\beta+8^\beta+16^\beta)}{8^\beta}\cdot\varepsilon, \quad x \in X.
    \end{equation}

    Accoding to Lemma \ref{the key lemma}, we can easily get
    \begin{eqnarray}\label{f(2x)inequ for addtive}
        \begin{aligned}
            &\left\|f(2 x)-\frac{2^{n}+1}{2 \cdot 4^{n}} f\left(2^{n+1} x\right)+\frac{2^{n}-1}{2 \cdot 4^{n}} f\left(-2^{n+1} x\right)\right\|\\
            &\leq\left( \sum_{n=1}^{\infty}{\left[ \left( \frac{2^n+1}{2\cdot 4^n} \right) ^{\beta}+\left( \frac{2^n-1}{2\cdot 4^n} \right) ^{\beta} \right] +1} \right) \cdot \frac{1+2^{\beta}+4^{\beta}+8^{\beta}+16^{\beta}}{8^{\beta}}\cdot \varepsilon
        \end{aligned}
    \end{eqnarray}
    for $x\in X$ and $n \in \mathbb{N}$. Moreover, for each $x \in X$ the sequence
    $$
    g_{n}(x):=\frac{2^{n}+1}{2 \cdot 4^{n}} f\left(2^{n} x\right)-\frac{2^{n}-1}{2 \cdot 4^{n}} f\left(-2^{n} x\right), \quad n \in \mathbb{N}
    $$
    is convergent in $Y$.

    Hence, the mapping $g : X \rightarrow Y$ can be defined as:
    \begin{equation*}
        g(x):=\lim _{n \rightarrow \infty} g_{n}(x)
    \end{equation*}
    for all $x \in X$.
    Combining with (\ref{f(2x)inequ for addtive}) we have
    \begin{equation*}
        \begin{aligned}
            &\Vert f(2x) - g(2x)\Vert \\
            &\leq  \left( \sum_{n=1}^{\infty}{\left[ \left( \frac{2^n+1}{2\cdot 4^n} \right) ^{\beta}+\left( \frac{2^n-1}{2\cdot 4^n} \right) ^{\beta} \right] +1} \right) \cdot \frac{1+2^{\beta}+4^{\beta}+8^{\beta}+16^{\beta}}{8^{\beta}}\cdot\varepsilon
        \end{aligned}
    \end{equation*}
    with $x \in X.$

    In order to prove that $g$ is orthogonally additive observe first that for
    $x, y \in X$ such that $x \perp y$ and $n \in N, n > 1$ we have
    $$
    \begin{aligned}&\quad\left\|2 g_{n}\left(\frac{x+y}{2}\right)-g_{n}(x)-g_{n}(y)\right\| \\
        &= \bigg\| \frac{2^{n}+1}{2 \cdot 4^{n}}\cdot 2f\left(\frac{2^{n}(x+y)}{2}\right)-\frac{2^{n}-1}{2 \cdot 4^{n}} \cdot 2f\left(\frac{-2^{n}(x+y)}{2}\right) \\
        &\quad-\frac{2^{n}+1}{2 \cdot 4^{n}} f\left(2^{n} x\right)+\frac{2^{n}-1}{2 \cdot 4^{n}} f\left(-2^{n} x\right)-\frac{2^{n}+1}{2 \cdot 4^{n}} f\left(2^{n} y\right)+\frac{2^{n}-1}{2 \cdot 4^{n}} f\left(-2^{n} y\right) \bigg\| \\
        &=\bigg \| \frac{2^{n}+1}{2 \cdot 4^{n}} \cdot \left[2 f\left(\frac{2^{n}(x+y)}{2}\right)-f\left(2^{n} x\right)-f\left(2^{n} y\right)\right] \\
        &\quad-\frac{2^{n}-1}{2 \cdot 4^{n}} \cdot \left[2 f\left(\frac{2^{n}(-x-y)}{2}\right)-f\left(-2^{n} x\right)-f\left(-2^{n} y\right)\right]\bigg \| \\
        & \leq  \left(\frac{2^{n}+1}{2 \cdot 4^{n}}\right)^\beta\left\|2f\left(\frac{2^{n}(x+y)}{2}\right)-f\left(2^{n} x\right)-f\left(2^{n} y\right)\right\| \\
        &\quad+\left(\frac{2^{n}-1}{2 \cdot 4^{n}}\right)^\beta\left\|2f\left(\frac{2^{n}(-x-y)}{2}\right)-f\left(-2^{n} x\right)-f\left(-2^{n} y\right)\right\| \\
        & \leq  \left[\left(\frac{2^{n}+1}{2 \cdot 4^{n}}\right)^\beta+\left(\frac{2^{n}-1}{2 \cdot 4^{n}}\right)^\beta\right]\cdot \varepsilon.
    \end{aligned}
    $$
    Moreover, letting $n \rightarrow \infty$,  we get (\ref{orth implies Jensen additive}).

    Now, we show the uniqueness of $g$. Assuming $g^{\prime}$ as another mapping satisfying
    (\ref{orth implies Jensen additive}) and (\ref{dist(f,g) for Jensen additive})  that yields:
\begin{equation*}
    \begin{aligned}
        &\quad\left\|g(x)-g^{\prime}(x)\right\|\\
        &\leq\|g(x)-f(x)\|+\left\|g^{\prime}(x)-f(x)\right\|\\
        &\leq 2\left( \sum_{n=1}^{\infty}{\left[ \left( \frac{2^n+1}{2\cdot 4^n} \right) ^{\beta}+\left( \frac{2^n-1}{2\cdot 4^n} \right) ^{\beta} \right] +1} \right) \cdot \frac{1+2^{\beta}+4^{\beta}+8^{\beta}+16^{\beta}}{8^{\beta}}\cdot \varepsilon
    \end{aligned}
\end{equation*}
    for all $x\in 2X$.

    On the other hand, the mapping $g-g^{\prime}$ satisfies (\ref{orth implies Jensen additive}) and thus, in
    particular, (\ref{epsilon Jensen additive}) with $\varepsilon= 0$. By applying (\ref{f(2x)inequ for addtive}) to $g-g^{\prime}$ we see that
    \begin{equation*}
        g(2 x)-g^{\prime}(2 x)=\frac{2^{n}+1}{2 \cdot 4^{n}}\left[g\left(2^{n+1} x\right)-g^{\prime}\left(2^{n+1} x\right)\right]-\frac{2^{n}-1}{2 \cdot 4^{n}}\left[g\left(-2^{n+1} x\right)-g^{\prime}\left(-2^{n+1} x\right)\right]
    \end{equation*}
    and  therefore
    \begin{equation*}
    \begin{aligned}
    &\left\|g(2 x)-g^{\prime}(2 x)\right\|\\
    \leq & \left(\frac{2^{n}+1}{2 \cdot 4^{n}}\right)^{\beta}\left\|g\left(2^{n+1} x\right)-g^{\prime}\left(2^{n+1} x\right)\right\|+\left(\frac{2^{n}-1}{2 \cdot 4^{n}}\right)^{\beta}\left\|g\left(-2^{n+1} x\right)-g^{\prime}\left(-2^{n+1} x\right)\right\| \\
    \leq & \left[\left(\frac{2^{n}+1}{2 \cdot 4^{n}}\right)^\beta+\left(\frac{2^{n}-1}{2 \cdot 4^{n}}\right)^\beta\right]\cdot 2\left( \sum_{n=1}^{\infty}{\left[ \left( \frac{2^n+1}{2\cdot 4^n} \right) ^{\beta}+\left( \frac{2^n-1}{2\cdot 4^n} \right) ^{\beta} \right] +1} \right) \cdot\\ &\frac{1+2^{\beta}+4^{\beta}+8^{\beta}+16^{\beta}}{8^{\beta}}\cdot \varepsilon
    \end{aligned}
    \end{equation*}
    for $x \in X$.

    Combining the both inequalities, we  can easily get the thesis.
\end{proof}

By the same method, we can also obtain the stability result for different target spaces as the following corollary, where the space $Y$ is equipped with quasi-norm.

\begin{corollary}
    Let $X$ be an Abelian group, and $Y$ be a quasi-Banach space.
    For $\varepsilon\geq 0$, assume $f: X\rightarrow Y$ be a mapping such that for all
    $x, y \in X$ one has
    \begin{gather*}
        x \perp y \quad \text { implies } \quad\|2f\left(\frac{x+y}{2}\right)-f(x)-f(y)\| \leq \varepsilon\\
        \text{and }\quad \|f(x)+f(-x)\|\leq \varepsilon.
    \end{gather*}

    Then there exists a mapping $g: X\rightarrow Y$ such that
    \begin{equation*}
        x \perp y \quad \text { implies } \quad 2g\left(\frac{x+y}{2}\right)=g(x)+g(y)
    \end{equation*}
    and
    \begin{equation*}\label{dist(f,g) for Jensen additive}
        \Vert f(x)-g(x)\Vert  \leq \left( \sum_{n=1}^{\infty}{\left[ \left( \frac{2^n+1}{2\cdot 4^n} \right) ^{p}+\left( \frac{2^n-1}{2\cdot 4^n} \right) ^{p} \right] +1} \right)^{\frac{1}{p}} \cdot \frac{(1+2^{p}+4^{p}+8^{p}+16^{p})^{\frac{1}{p}}}{8}\cdot \varepsilon
    \end{equation*}
    for all $x \in 2X = \{2x : x \in X\}$. Moreover, the mapping $g$ is unique on the
    set $2X$.
\end{corollary}
\begin{proof}
    Let $\|\cdot\|_p=\|\cdot\|^p$, then it is obviously that $(Y, \|\cdot\|_p)$ is $p$-homogeneous, we obtain
    \begin{equation*}
        \begin{gathered}
            x \perp y \quad \text { implies } \quad\|2f(\frac{x+y}{2})-f(x)-f(y)\|_p \leq \varepsilon^p\\
            and \quad \|f(x)+f(-x)\|_p\varepsilon^p.
        \end{gathered}
    \end{equation*}
    According to Theorem \ref{main theorem for additive}, we obtain that there exists a mapping $g: X\rightarrow Y$ such that
    $$
    x \perp y \quad \text { implies } \quad 2g(\frac{x+y}{2})=g(x)+g(y)
    $$
    and
    \begin{equation*}
        \Vert f(x)-g(x)\Vert_p\leq \left( \sum_{n=1}^{\infty}{\left[ \left( \frac{2^n+1}{2\cdot 4^n} \right) ^{\beta}+\left( \frac{2^n-1}{2\cdot 4^n} \right) ^{\beta} \right] +1} \right) \cdot \frac{1+2^{\beta}+4^{\beta}+8^{\beta}+16^{\beta}}{8^{\beta}}\cdot\varepsilon^p
    \end{equation*}
    for all $x \in 2X = \{2x : x \in X\}$. Moreover, the mapping $g$ is unique on the
    set $2X$ and the claim follows.
\end{proof}

%% file: JensenQuadractic.tex
\section{Stability of the orthogonally Jensen quadractic functional equation}

\quad Applying and extending some ideas from \cite{2010otso}, we deal with the conditional stability problem for (\ref{Jensen quadratic}).

The main result of this section: In this section, let $X$ be an Abelian group and let $\perp$ be a binary relation defined on $X$ with the properties:

($a'$) for all $x\in X, 0\perp x$ and $x\perp 0$;

($b'$) if $x, y \in X$ and $x\perp y$, then $-x\perp -y$ and $2x\perp 2y$.

\begin{theorem}
    Let $X$ be an Abelian group, and $Y$ be a $\beta$-homogeneous $F$-space.
    For $\varepsilon\geq 0$, assume $f: X\rightarrow Y$ be a mapping such that for all
    $x, y \in X$ one has
    \begin{eqnarray}\label{epsilon Jensen quadractic}
        \begin{aligned}
            x \perp y \quad \text { implies } \quad \|2 f\left(\frac{x+y}{2}\right)+2 f\left(\frac{x-y}{2}\right) -f(x)-f(y)\| \leq \varepsilon
        \end{aligned}
    \end{eqnarray}
    Then there exists a mapping $g: X\rightarrow Y$ such that
    \begin{eqnarray}\label{orth implies Jensen quadractic}
        \begin{aligned}
            x \perp y \quad \text { implies } \quad 2g\left(\frac{x+y}{2}\right)+2g\left(\frac{x-y}{2}\right)=g(x)+g(y)
        \end{aligned}
    \end{eqnarray}
    and
    \begin{equation}\label{dist(f,g) for Jensen quadractic}
            \Vert f(x)-g(x)\Vert\leq \left( \sum_{n=1}^{\infty}{\left[ \left( \frac{2^n+1}{2\cdot 4^n} \right) ^{\beta}+\left( \frac{2^n-1}{2\cdot 4^n} \right) ^{\beta} \right]}+1 \right) \cdot \frac{1+2^{\beta}+2^{1-\beta}+2^{1-2\beta}}{8^{\beta}}\cdot \varepsilon
    \end{equation}
    for all $x \in 2X = \{2x : x \in X\}$. Moreover, the mapping $g$ is unique on the
    set $2X$.
\end{theorem}

\begin{proof} For all $x\in X$. By ($a'$), we have $0\perp 0$, $0\perp x$ and $x\perp 0$, so according to $0\perp 0$ and (\ref{epsilon Jensen quadractic}) we have
    $$
    \|2f(0)+2f(0)-f(0)-f(0)\|\leq \varepsilon
    $$
    and then we have
    \begin{equation}\label{f(0) leq epsilon/2}
        \|f(0)\|\leq \frac{\varepsilon}{2^\beta}.
    \end{equation}
    According to $0\perp x$, we can write
    $$
    \|2f\left(\frac{x}{2}\right)+2f\left(-\frac{x}{2}\right)-f(0)-f(x)\|\leq \varepsilon
    $$
    it is easy to get
    \begin{equation}\label{estimate of 2f(x/2)+2f(-x/2)-f(x)}
        \|2f\left(\frac{x}{2}\right)+2f\left(-\frac{x}{2}\right)-f(x)\|\leq (\frac{1}{2^\beta}+1)\varepsilon
    \end{equation}
    Since $x\perp 0$, we have
    $$
    \|2f\left(\frac{x}{2}\right)+2f\left(\frac{x}{2}\right)-f(x)-f(0)\|\leq \varepsilon
    $$
    combined with (\ref{f(0) leq epsilon/2}), it is obvious that
    \begin{equation}\label{estime of 4f(x/2)-f(x)}
        \|4f\left(\frac{x}{2}\right)-f(x)\|\leq (\frac{1}{2^\beta}+1)\varepsilon.
    \end{equation}

    On the other hand, by (\ref{estimate of 2f(x/2)+2f(-x/2)-f(x)}) and (\ref{estime of 4f(x/2)-f(x)}), we have
    $$
    \|2f\left(\frac{x}{2}\right)+2f\left(-\frac{x}{2}\right)-4f\left(\frac{x}{2}\right)\|\leq 2(\frac{1}{2^\beta}+1)\varepsilon
    $$
    so we have
    $$
    \|f\left(\frac{x}{2}\right)-f\left(-\frac{x}{2}\right)\|\leq 2^{1-\beta}(\frac{1}{2^\beta}+1)\varepsilon
    $$
    this means that
    \begin{equation*}
        \|f(x)-f(-x)\|\leq 2^{1-\beta}(\frac{1}{2^\beta}+1)\varepsilon.
    \end{equation*}
    Then it is easy to get
    $$
    \begin{aligned}&\| 3 f(4 x)-8 f(2 x)-f(-4 x)) \|\\
        &=\|\left[ 2f\left( 4x \right) -8f\left( 2x \right) \right] +\left[ f\left( 4x \right) -f\left( -4x \right) \right] \| \\
        &\leq 2^{\beta}\left( \frac{1}{2^{\beta}}+1 \right) \varepsilon +2^{1-\beta}\left( \frac{1}{2^{\beta}}+1 \right) \varepsilon \\
        &=(1+2^{\beta}+2^{1-\beta}+2^{1-2\beta})\varepsilon.
    \end{aligned}
    $$
    This proves that
    \begin{eqnarray}\label{183ineq for quadratic}
        \begin{aligned}
            \left\|f(2 x)-\frac{3}{8} f(4 x)+\frac{1}{8} f(-4 x)\right\| \leq \frac{1+2^{\beta}+2^{1-\beta}+2^{1-2\beta}}{8^\beta}\varepsilon, \quad x \in X.
        \end{aligned}
    \end{eqnarray}

    Accoding to Lemma \ref{the key lemma}, we can easily get
    \begin{equation}\label{f(2x)inequ for quadratic}
        \begin{aligned}
            &\left\|f(2 x)-\frac{2^{n}+1}{2 \cdot 4^{n}} f\left(2^{n+1} x\right)+\frac{2^{n}-1}{2 \cdot 4^{n}} f\left(-2^{n+1} x\right)\right\| \\ &\leq \left( \sum_{n=1}^{\infty}{\left[ \left( \frac{2^n+1}{2\cdot 4^n} \right) ^{\beta}+\left( \frac{2^n-1}{2\cdot 4^n} \right) ^{\beta} \right]}+1 \right) \cdot \frac{1+2^{\beta}+2^{1-\beta}+2^{1-2\beta}}{8^{\beta}} \varepsilon.
        \end{aligned}
    \end{equation}
for $x\in X$ and $n \in \mathbb{N}$. Moreover, for each $x \in X$ the sequence
$$
g_{n}(x):=\frac{2^{n}+1}{2 \cdot 4^{n}} f\left(2^{n} x\right)-\frac{2^{n}-1}{2 \cdot 4^{n}} f\left(-2^{n} x\right), \quad n \in \mathbb{N}
$$
is convergent in $Y$.

    Hence, the mapping $g : X \rightarrow Y$ can be defined as:
    \begin{equation*}
        g(x):=\lim _{n \rightarrow \infty} g_{n}(x)
    \end{equation*}
    for all $x \in X$.
    Combining with (\ref{f(2x)inequ for quadratic}) we have
    $$\Vert f(2x) - g(2x)\Vert \leq  \left( \sum_{n=1}^{\infty}{\left[ \left( \frac{2^n+1}{2\cdot 4^n} \right) ^{\beta}+\left( \frac{2^n-1}{2\cdot 4^n} \right) ^{\beta} \right]}+1 \right) \cdot \frac{1+2^{\beta}+2^{1-\beta}+2^{1-2\beta}}{8^{\beta}}\varepsilon, \quad x \in X.$$
    In order to prove that $g$ is orthogonally additive observe first that for
    $x, y \in X$ such that $x \perp y$ and $n \in N, n > 1$ we have
    \begin{equation*}
        \begin{aligned}&\quad\left\|2 g_{n}\left(\frac{x+y}{2}\right)+2 g_{n}\left(\frac{x-y}{2}\right)-g_{n}(x)-g_{n}(y)\right\| \\
            &= \bigg\| \frac{2^{n}+1}{2 \cdot 4^{n}}\cdot 2f\left(\frac{2^{n}(x+y)}{2}\right)-\frac{2^{n}-1}{2 \cdot 4^{n}} \cdot 2f\left(\frac{-2^{n}(x+y)}{2}\right) \\
            &\quad +\frac{2^{n}+1}{2 \cdot 4^{n}}\cdot 2f\left(\frac{2^{n}(x-y)}{2}\right)-\frac{2^{n}-1}{2 \cdot 4^{n}} \cdot 2f\left(\frac{-2^{n}(x-y)}{2}\right)\\
            &\quad-\frac{2^{n}+1}{2 \cdot 4^{n}} f\left(2^{n} x\right)+\frac{2^{n}-1}{2 \cdot 4^{n}} f\left(-2^{n} x\right)-\frac{2^{n}+1}{2 \cdot 4^{n}} f\left(2^{n} y\right)+\frac{2^{n}-1}{2 \cdot 4^{n}} f\left(-2^{n} y\right) \bigg\| \\
            &=\bigg \| \frac{2^{n}+1}{2 \cdot 4^{n}} \cdot \left[2 f\left(\frac{2^{n}(x+y)}{2}\right)+2 f\left(\frac{2^{n}(x-y)}{2}\right)-f\left(2^{n} x\right)-f\left(2^{n} y\right)\right] \\
            &\quad-\frac{2^{n}-1}{2 \cdot 4^{n}} \cdot \left[2 f\left(\frac{2^{n}(-x-y)}{2}\right)+2 f\left(\frac{2^{n}(-x+y)}{2}\right)-f\left(-2^{n} x\right)-f\left(-2^{n} y\right)\right]\bigg \| \\
            & \leq  \left(\frac{2^{n}+1}{2 \cdot 4^{n}}\right)^\beta\left\|2f\left(\frac{2^{n}(x+y)}{2}\right)+2 f\left(\frac{2^{n}(x-y)}{2}\right)-f\left(2^{n} x\right)-f\left(2^{n} y\right)\right\| \\
            &\quad+\left(\frac{2^{n}-1}{2 \cdot 4^{n}}\right)^\beta\left\|2f\left(\frac{2^{n}(-x-y)}{2}\right)+2f\left(\frac{2^{n}(-x+y)}{2}\right)-f\left(-2^{n} x\right)-f\left(-2^{n} y\right)\right\| \\
            & \leq  \left[ \left( \frac{2^n+1}{2\cdot 4^n} \right) ^{\beta}+\left( \frac{2^n-1}{2\cdot 4^n} \right) ^{\beta} \right] \varepsilon.
        \end{aligned}
    \end{equation*}
    Moreover, letting $n \rightarrow \infty$,  we get (\ref{orth implies Jensen quadractic}).

    Now, we show the uniqueness of $g$. Assuming $g^{\prime}$ as another   mapping satisfying
    (\ref{orth implies Jensen quadractic}) and (\ref{dist(f,g) for Jensen quadractic})  that yields:
    \begin{equation*}
        \begin{aligned}
            &\left\|g(x)-g^{\prime}(x)\right\|\\ &\leq\|g(x)-f(x)\|+\left\|g^{\prime}(x)-f(x)\right\|\\
            &\leq 2\left( \sum_{n=1}^{\infty}{\left[ \left( \frac{2^n+1}{2\cdot 4^n} \right) ^{\beta}+\left( \frac{2^n-1}{2\cdot 4^n} \right) ^{\beta} \right]}+1 \right) \cdot \frac{1+2^{\beta}+2^{1-\beta}+2^{1-2\beta}}{8^{\beta}} \varepsilon
        \end{aligned}
    \end{equation*}
    for all $x\in 2X$.

    On the other hand, the mapping $g-g^{\prime}$ satisfies (\ref{orth implies Jensen quadractic}) and thus, in
    particular, (\ref{epsilon Jensen quadractic}) with $\varepsilon= 0$. By applying (\ref{f(2x)inequ for quadratic}) to $g-g^{\prime}$ we see that
    $$
    \begin{aligned} g(2 x)-g^{\prime}(2 x)=& \frac{2^{n}+1}{2 \cdot 4^{n}}\left[g\left(2^{n+1} x\right)-g^{\prime}\left(2^{n+1} x\right)\right] \\ &-\frac{2^{n}-1}{2 \cdot 4^{n}}\left[g\left(-2^{n+1} x\right)-g^{\prime}\left(-2^{n+1} x\right)\right] \end{aligned}
    $$
    and  therefore
    \begin{equation*}
    \begin{aligned}
        &\left\|g(2 x)-g^{\prime}(2 x)\right\|\\
        \leq & \left(\frac{2^{n}+1}{2 \cdot 4^{n}}\right)^{\beta}\left\|g\left(2^{n+1} x\right)-g^{\prime}\left(2^{n+1} x\right)\right\|+\left(\frac{2^{n}-1}{2 \cdot 4^{n}}\right)^{\beta}\left\|g\left(-2^{n+1} x\right)-g^{\prime}\left(-2^{n+1} x\right)\right\| \\
        \leq & \left[\left(\frac{2^{n}+1}{2 \cdot 4^{n}}\right)^\beta+\left(\frac{2^{n}-1}{2 \cdot 4^{n}}\right)^\beta\right]\cdot 2\left( \sum_{n=1}^{\infty}{\left[ \left( \frac{2^n+1}{2\cdot 4^n} \right) ^{\beta}+\left( \frac{2^n-1}{2\cdot 4^n} \right) ^{\beta} \right] +1} \right) \cdot\\ &\frac{1+2^{\beta}+2^{1-\beta}+2^{1-2\beta}}{8^{\beta}} \varepsilon
    \end{aligned}
\end{equation*}
    for $x \in X$.

    Combining the both inequalities, we  can easily get the thesis.
\end{proof}

By the same method, we can also obtain the stability result for different target spaces as the following corollary, where the space $Y$ is equipped with quasi-norm.

\begin{corollary}
    Let $X$ be an Abelian group, and $Y$ be a quasi-Banach space.
    For $\varepsilon\geq 0$, assume $f: X\rightarrow Y$ be a mapping such that for all
    $x, y \in X$ one has
    \begin{equation*}
        x \perp y \quad \text { implies } \quad \|2 f\left(\frac{x+y}{2}\right)+2 f\left(\frac{x-y}{2}\right) -f(x)-f(y)\| \leq \varepsilon
    \end{equation*}

    Then there exists a mapping $g: X\rightarrow Y$ such that
    \begin{equation*}
        x \perp y \quad \text { implies } \quad 2g\left(\frac{x+y}{2}\right)+2g\left(\frac{x-y}{2}\right)=g(x)+g(y)
    \end{equation*}
    and
    \begin{equation*}\label{dist(f,g) for Jensen additive}
        \Vert f(x)-g(x)\Vert  \leq \left( \sum_{n=1}^{\infty}{\left[ \left( \frac{2^n+1}{2\cdot 4^n} \right) ^{p}+\left( \frac{2^n-1}{2\cdot 4^n} \right) ^{p} \right]}+1 \right)^{\frac{1}{p}} \cdot \frac{(1+2^{p}+2^{1-p}+2^{1-2p})^\frac{1}{p}}{8}\cdot \varepsilon
    \end{equation*}
    for all $x \in 2X = \{2x : x \in X\}$. Moreover, the mapping $g$ is unique on the
    set $2X$.
\end{corollary}
\begin{proof}
    Let $\|\cdot\|_p=\|\cdot\|^p$, then it is obviously that $(Y, \|\cdot\|_p)$ is $p$-homogeneous, we obtain
    \begin{equation*}
            x \perp y \quad \text { implies } \quad\|2 f\left(\frac{x+y}{2}\right)+2 f\left(\frac{x-y}{2}\right) -f(x)-f(y)\|_p \leq \varepsilon^p.
    \end{equation*}
    According to Theorem \ref{main theorem for additive}, we obtain that there exists a mapping $g: X\rightarrow Y$ such that
    \begin{equation*}
        x \perp y \quad \text { implies } \quad 2g\left(\frac{x+y}{2}\right)+2g\left(\frac{x-y}{2}\right)=g(x)+g(y)
    \end{equation*}
    and
    \begin{equation*}
        \Vert f(x)-g(x)\Vert_p\leq \left( \sum_{n=1}^{\infty}{\left[ \left( \frac{2^n+1}{2\cdot 4^n} \right) ^{p}+\left( \frac{2^n-1}{2\cdot 4^n} \right) ^{p} \right] +1} \right) \cdot \frac{1+2^{p}+2^{1-p}+2^{1-2p}}{8^p}\cdot\varepsilon^p
    \end{equation*}
    for all $x \in 2X = \{2x : x \in X\}$. Moreover, the mapping $g$ is unique on the
    set $2X$ and the claim follows.
\end{proof}